\documentclass{article}
\usepackage{graphicx} 

\title{Constructing strong starters of orders $3p$: triplication with SAT solver}

\author{Oleg Ogandzhanyants$^\dag$, Sergey Sadov$^\ddag$, Margo Kondratieva$^\dag$%
\footnote{Corresponding author. E-mail:{ \tt mkondra@mun.ca}}
}
\date{} 
\usepackage{amssymb,amsmath,amsthm}
\usepackage{hyperref}

\usepackage{listings}
\usepackage{xcolor}
\usepackage{pgfplotstable}
\usepackage{pgfplots}
\pgfplotsset{compat=1.13}

\newcommand{\ZZ}{\mathbb{Z}}
\newcommand{\MExp}[1]{\mathbb{E}#1} 
\newcommand{\Stdev}[1]{\sigma}  
\newcommand{\sumset}{\hat{S}}
\newcommand{\notdiv}{\nmid}
\newcommand{\zsat}{{\tt z3}\ }
\newcommand{\mmod}{\!\! \mod}
\newcommand{\ind}{\mathrm{ind}}
\newcommand{\ab}[1]{\langle #1\rangle}

\newtheorem{lemma}{Lemma}

\newtheorem{theorem}{Theorem}

\theoremstyle{definition}
\newtheorem*{definition}{Definition}

\begin{document}
	
	\maketitle

{\small
$^\dag$  Dept. of Mathematics and Statistics, Memorial University, St. John's NL  A1C~5S7, Canada

$^\ddag$  Moscow, Russia
}

\begin{abstract}
A novel approach to building strong starters in cyclic groups of orders $n$ divisible by 3 from starters of smaller
orders is presented. A strong starter in $\ZZ_n$ ($n$ odd) is a partition of the set $\{1,2,\dots,n-1\}$
into pairs $\{a_i,b_i\}$ such that all pair sums $a_i+b_i$ are distinct and nonzero modulo $n$
and all differences $\pm(a_i-b_i)$ are distinct and nonzero modulo $n$. A special interest to strong starters
of odd orders divisible by 3 is motivated by Horton's conjecture which claims that such starters exist 
(except when $n=3$ or $9$) but remains unproven since 1989.

We begin with a strong starter of order $p$ coprime with 3 and 
describe an algorithm 
to obtain a Sudoku-type problem modulo 3 whose solution, if exists,
yields a strong starter of order $3p$. The process leading from the original to the final starter is called {\em triplication}.

Besides theoretical aspects of the construction, practicality of this approach is demonstrated.
A general-purpose constraint-satisfaction (SAT) solver z3 is used to solve the Sudoku-type problem;
various performance statistics are presented.

\end{abstract}
    
	\section{Introduction}
	
	We discuss a new approach for constructing strong starters
	(special type of combinatirial designs) in the case 
    of orders 
    $n$ divisible by 3, where their existence has not been asserted theoretically. 	
	
	A \textit{starter} $S$ in an additive abelian group $G$ of odd order $n$ is a partition of the set $G^*=G\setminus\{0\}$ into $k=\frac{n-1}2$ pairs $\{\{a_i,b_i\}\}_{i=1}^{k}$ such that the differences $\{\pm(a_i-b_i), i=1,...,k\}$, comprise $G^*$. 
	See \cite{Dinitz07} for a general reference.
	
	Starters exist in any additive abelian group of odd order $n\ge3$. For example, the partition
	$\{\{x,-x\}\mid x\in G^*\}$ of $G^*$  is a starter in $G$.
	
	For a starter $S$ in $G$ of order $n=2k+1$, consider the set of sums in pairs, $\sumset=\{a_i+b_i\mid\{a_i,b_i\}\in S\}$. 
	If $0\notin \sumset$ and $|\sumset|=k$, the starter $S$ is called {\em strong}. 

    Any abelian group $G$ of odd order $n\geq 7$ for which $3\notdiv n$ admits a strong starter \cite{Horton89}. 
	Horton \cite{Horton89} put forward a conjecture, supported by computational evidence, to the effect that strong starters exist in all abelian groups
	of odd orders except a few groups of orders $n=3,5,9,$ for which non-existence is easily verified.
		
	In this paper, $G$ will always be a cyclic group $\ZZ_n=\ZZ/n\ZZ$, its elements identified with
	integers in the range $0$ to $n-1$. We will for brevity say ``$S$ is a starter of order $n$'' meaning a starter in $\ZZ_n$.

The question on the existence of strong starters in $\ZZ_n$ with $n$ divisible by 3 is only partially settled.
According to Horton's summary in \cite{Horton89}, theoretical results supplemented by Dinitz and Stinson's computations  \cite{DinitzStinson81, Stinson85} 
guarantee the existence of strong starters of odd orders $n=3^r p$, $r>0$, only if either
$n< 1000$, $n\notin\{3,9\}$, or if 
$\gcd(p,3)=1$ and $3^r<1000$, $3^r\notin\{3,9\}$ (that is, $r\in\{3,4,5,6\}$).%
\footnote{Our recent computations, to be described in a separate paper, extend the upper boundary in the 
inequalities for $n$ or $3^r$ from 1000 to $2^{16}-1$.}
In particular, if $n=3p>1000$, where $p$ is coprime with $3$, the existence of strong starters of order $n$
is not asserted in literature.

This gap in theoretical knowledge,
which, 
according to a recent survey by Stinson \cite{Stinson23}, 
apparently still persisted in 2022,
stimulates interest to ways of obtaining starters of orders divisible by 3.

We will describe a seemingly novel method of construction of strong starters in $\ZZ_n$, where $n=3p$,
$\gcd(3,p)=1$. We call it {\em triplication} (of a strong starter). The method takes a strong starter of order $p$ 
as a ``base'' and reduces the problem of construction of
a strong starter of order $3p$ to a certain system of linear equations and inequalities  modulo 3. Below, we call it a {\em Sudoku-type problem $\mmod 3$}, or a {\em modular Sudoku problem}. 
The required starter of order $3p$ is recovered from the base starter's {extension} (defined in our method) and a solution of the modular Sudoku problem
by means of the Chinese Remainder Theorem (CRT).

The said modular Sudoku problem is a new object of research,
which invites such natural questions as a possibility to
prove the existence of a solution theoretically and numerical explorations.  
Here, as a first demonstration of the practicality of the new method with
minimal programming effort, 
we present an approach relying on a general purpose constraint-satisfaction
(SAT) solver.%
	\footnote{SAT solver is a computer program which aims to solve the  satisfiability problem (SAT). On input of a set of equations, a SAT solver outputs whether the formulas are satisfiable. ``Since the introduction of algorithms for SAT in the 1960s, modern SAT solvers have grown into complex software artifacts involving a large number of heuristics and program optimizations to work efficiently.''\cite{Wiki-SAT}}
Specifically, the SAT library \zsat by Microsoft was used --- for
no other reason than its ready availability and a simple Python interface.

We plan to describe a specialized algorithm for solving the modular Sudoku problem, which by far outperforms the SAT solver, in a separate paper. 

\section{Ingredients of triplication by example}
\label{sec:example}

Let us consider a strong starter of the least possible order $p=7$,
$$
T=\{\{2,\ 3\},\ \{4,\ 6\},\ \{5,\ 1\}\}.
$$
Indeed, $T$ is clearly a 2-partition of $\ZZ_7^*$; it is a starter because $\{\pm(2-3),\ \pm(4-6),\ \pm(1-5)\}\mmod 7=\{\pm1,\ \pm2,\ \pm3\}=Z_7^*$. It is strong: $\hat{T}=\{2+3,\ 4+6,\ 1+5 \}= \{5,\ 3,\ 6\}\subset \ZZ_7^*$ and $|\hat{T}| =(7-1)/2=3$. 
		
In order to motivate the triplication method, we 
will scrutinize a final answer --- a starter $S$ of order 21, which in the actual course of events would have been reconstructed 
from $T$ and a solution of the appropriate Sudoku-type $\mmod 3$ problem. But here,
we work backwards, from $S$ to $T$. In a series of exercises, we will exhibit
all types of constraints embedded in the modular Sudoku problem.

\smallskip
Well, so by an act of black magic we've got a strong starter $S$ of order 21:
$$
\begin{array}{rcl}
S&=& \{\{11,\ 18\}, \ \{9,\ 17\},\{13,\ 14\},\ \{8,\ 2\},\ \{4,20\},\\ 
 &&	\;\,\{15,3\}, \{5,\ 7\},\ \{1,\ 12\},\ \{19,\ 16\},\ \{6,\ 10\}\}.
\end{array}    
$$  
%
%
Consider two modular reductions of $S$.
The set of pairs of $S$ (or {\em pairing} for short) $\mmod 3$ is
$$
\Sigma_3=\{\{2,\ 0\},\ \{0,\ 2\},\{1,\ 2\},\ \{2,\ 2\},\ \{1,2\},\ \{0,0\},\ \{2,\ 1\},\ \{1,\ 0\},\ \{1,\ 1\},\ \{0,\ 1\}\}.
$$
The pairing $\mmod 7$ corresponding to $S$ is

$$
\Sigma_7=\{\{4,\ 4\},\ \{2,\ 3\},\{6,\ 0\},\ \{1,\ 2\},\ \{4,6\},\ \{1,3\},\ \{5,\ 0\},\ \{1,\ 5\},\ \{5,\ 2\},\ \{6,\ 3\}\}.
$$

The starter $S$ can be uniquely reconstructed from $\Sigma_3$ and $\Sigma_7$  
by CRT. 
	
It is convenient to write the pairings $\Sigma_3$ and $\Sigma_7$ in table format as follows:

\bigskip

\begin{tabular}{c c c}
    $\Sigma_3$ && $\Sigma_7$\\[1ex]
	\begin{tabular}{|cc||cc||cc|}
		\hline 
		&&$2,$ & $0$ & &  \\
		\hline
		$0$, & $2$  & $1$, & $2$& $2$, & $2$\\ 
		\hline
		$1$, & $2$ & $0$, & $0$& $2$, & $1$ \\ 
		\hline
		$1$, & $0$ & $1$, & $1$& $0$, & $1$ \\ 
		\hline
	\end{tabular}
&	
$\qquad$ &
    \begin{tabular}{|cc||cc||cc|}
		\hline 
		&&$4,$ & $4$ & &  \\
		\hline
		$2$, & $3$  & $6$, & $0$& $1$, & $2$\\ 
		\hline
		$4$, & $6$ & $1$, & $3$& $5$, & $0$ \\ 
		\hline
		$1$, & $5$ & $5$, & $2$& $6$, & $3$ \\ 
		\hline
	\end{tabular}
\end{tabular}

\bigskip

From now on, we treat $\Sigma_3$ and $\Sigma_7$ as pairings or as tables as convenient in the context.

Note that the first column of the table $\Sigma_7$, read top to bottom, is that very starter 
$T$ of order 7 seen earlier in this section.

There is no analog of this fact pertaining to the table $\Sigma_3$, since the set $\ZZ_3^*=\{1,2\}$
is too small to be partitioned into 3 pairs.

The nature of the 2nd and 3rd columns of the table $\Sigma_7$ remains at this stage a mystery.

Leaving a formal description of relations between the elements in each table and relations
between the elements in different tables to later sections, we recommend the reader or his/her (grand)child to do
some simple arithmetical exercises. (Familiarity with elementary modular arithmetic is required.)

\smallskip
(a1) Check that all differences $\mmod 3$ in each row of the table $\Sigma_3$ are distinct (e.g.
$0-2\equiv 1 \mmod 3, 1-2\equiv 2 \mmod 3, 2-2\equiv 0 \mmod 3 $).

(a2) Check that all differences $\mmod 7$ in each row of the table $\Sigma_7$ are equal (e.g.
$2-3\equiv 6-0\equiv 1-2 \mmod 7$).

\smallskip
(b0) Calculate all pair sums modulo 3 for the table $\Sigma_3$ and all pair sums modulo 7 for the table $\Sigma_7$.  

(b1) Find a cell in the table $\Sigma_7$ with sum $=0\mmod 7$ and check that in the corresponding cell of
the table $\Sigma_3$ the sum $\mmod 3$ is nonzero.

(b2) Check that whenever two cell sums $\mmod 7$ in the table $\Sigma_7$ are equal, the corresponding
cell sums $\mmod3$ for the table $\Sigma_3$ are distinct. For example, the cells $(2,3)$ and $(5,0)$
have equal pair sums $\mmod 7$. The corresponding cells in the table $\Sigma_3$ are $(0,2)$ and $(2,1)$,
and their sums $\mmod 3$ are distinct.

\smallskip
(c1) Note the positions of zeros in the table $\Sigma_7$ and observe that the corresponding places 
in the table $\Sigma_3$ are occupied by two distinct nonzero values, $1$ and $2$.

(c2) For each $c\in\{1,\dots,6\}$, note that the number $c$ appears in exactly three places in the table $\Sigma_7$,
and the corresponding positions in the table $\Sigma_3$ are occupied by three distinct numbers $0,1,2$.

\smallskip
(d) The structure of the table $\Sigma_7$ is a cryptographic challenge of a sort. 
Alice wants to pass three messages 
represented by pairs in the first column to her party Bob by radio. The messages are not for prying 
ears.
So Alice encrypts them before going live. She sends one message on Monday, repeats it on Tuesday;
a new message on Wednesday, repeated on Thursday; then a message on Friday, repeated on Saturday.
On odd days, the encrypted message aired by Alice and received by Bob and also by the curious interceptor 
Cindy is represented by the pair in the second column (same row as the original message); on even days,
the encryption is the pair in the third column.
This past week was lucky for Cindy as she had an agent in Alice's camp, through which she came in possession 
of all the original messages as well as the cipher key, the number repeated twice in the uppermost cell
of the table (in our case, the number $4$). Since Saturday night, the agent's fate has been unknown and there
is no warranty that Cindy ever has such a luxury again. She will still intercept encrypted messages aired by Alice.
The cipher key is likely to change next week. 
Being in Cindy's shoes, could you figure out the encryption scheme and continue to eavesdrop Alice's secrets?

\smallskip
(e) There are strong starters of order 21 other than $S$ whose reduction modulo 7 coincides with $\Sigma_7$.
For example,
$$
\begin{array}{rcl}
S' &=& \{\{4, 18\}, \ \{16, 10\}, \ \{20, 7\}, \ \{8, 9\}, \ \{11, 13\}, \\
&&	\;\,
\{1, 17\}, \ \{5, 14\}, \ \{15, 19\}, \ \{12, 2\}, \ \{6, 3\}\}.
\end{array}    
$$  
The reduction of $S'$ modulo 3,
$$
\Sigma'_3=\{\{1, 0\},\ \{1, 1 \},\ \{2, 1 \},\ \{2, 0 \},\ \{2, 1 \},\ \{1, 2 \},\ \{2, 2 \},\ \{0, 1 \},\ \{0, 2\},\ \{0, 0\}\}.
$$
is different from $\Sigma_3$ --- this had to be expected as otherwise, by CRT, $S$ and $S'$ would be identical. 

\medskip
Whether or not the reader had patience to do all or some of the proposed exercises,
we summarize essential points to be learned from this example. (Replace $7$ by $p$ and $21$ by $3p$
for a general picture. Here $p\ge 7$ coprime to 6.)

\begin{enumerate}
\item 
The resulting starter $S$ of order 21 can be uniquely recovered from the two pairings $\Sigma_3$ and $\Sigma_7$.

\item 
The first column of the table $\Sigma_7$ is a  copy of  the ``base starter'' $T$ of order 7.

\item 
The pair on top of the table $\Sigma_7$ consists of two identical values.
Such a value, or ``key'', is another ``free parameter''  of the
method (along with base starter).

\item 
The second and third columns of the table $\Sigma_7$ are uniquely determined by the base starter and the key.
The readers who figured out exercise (d) know the formulas; others will find them in the next section. 

\item
While the order of elements within pairs in the definition of a starter is unimportant,
our method requires ordered pairs. Indeed, the relations between
the tables for $\Sigma_3$ and $\Sigma_7$ dealt with in exercises (c1,2) depend on specific order of elements
in pairs.

\item 
The table modulo 3 is not arbitrary; its values must satisfy certain constraints.
The constraints can be split into three groups (cf.\ exercises of groups a,b,c). A common feature of these
constraints can be stated as follows. Whenever an {\bf equality}\ relation of certain special type 
between elements in the table modulo 7 is observed, the corresponding 
elements in the first table must satisfy an {\bf inequality}\ of similar arithmetical structure.

\item 
The table $\Sigma_3$ is not to be expected to be computable easily and by straightforward formulas given the outlined set-up.
An admissible filling of $\Sigma_3$ with numbers $0,1,2$, if it exists, is not to be expected to be unique.
Rather, the problem of determination of $\Sigma_3$ is a constraint-satisfaction problem modulo 3.

\item 
Finally, we note that a starter which comes out as a result of this procedure (that's what
the ``black magic'' leading to our starter $S$, as well as to $S'$ in exercise (e), 
was really about), cannot be just any strong starter of order 21. 
Implicitly, this follows from the fact that that there exist no strong starters of order 5,
but there are\footnote{precisely, 32 strong starters or 4 equivalence classes \cite{Vesa10}.} 
strong starters of order 15.

The statement that a certain $S$ is not a result of triplication needs clarification.
For instance,  we require that the set of pairs comprising the first column of the table $\Sigma_7$ must be a strong
starter of order 7. But this condition is not invariant under permutation of pairs in $S$, while
the notion of a starter is. 
An invariant way to detect that $S$ is not a result of triplication is mentioned at the end of Sec.~\ref{sec:theorems}.
\end{enumerate}


The rest of the paper is organized as follows.    

In Section~\ref{sec:sigma_p}, a construction of $\Sigma_p$ is revealed (that is, a solution to exercise (d)).  Sections~\ref{sec:weak} and ~\ref{sec:color} define two auxiliary tables derived from $\Sigma_p$. 
Section~\ref{sec:sigma_3} is devoted to setting up the modular Sudoku problem: 
the table $\Sigma_3$ of unknowns $\mmod 3$ is introduced and the constraints imposed on the unknowns
(those observed in the exercises (a1), (b*), (c*))
are formally described with help of the tables from Sec.~\ref{sec:weak}, \ref{sec:color}. 

Section~\ref{sec:theorems} is devoted to rigorous theory: we prove that a solution of the modular Sudoku problem yields a strong starter of order $3p$
(Theorem~\ref{thm:sudoku_to_starter}), establish a necessary condition of solvability (Theorem~\ref{thm:cond-key}), and observe an involutive symmetry of the set of solutions (Theorem~\ref{thm:sudoku-symmetry}).
Also, we mention a way (sufficient condition) to recognize that a given starter of order $3p$ cannot be obtained
by triplication.

Finally, in Section~\ref{sec:computations}
we report on our experience of solving the modular Sudoku problem with help of the SAT solver \zsat
and make some comments on comparative performance.
			
\section[The triplication table Sigma(p)]{The triplication table $\Sigma_p$}
\label{sec:sigma_p}

The construction to be described takes two input parameters:
\begin{itemize}
\item a 
starter $T$ of odd order $p$ coprime with 3;
\item a number $t\in\{0, 1,\dots,p-1\}$, which will be called the {\em key}.
\end{itemize}

The number of pairs in $T$ will be denoted
$$
q=\frac{p-1}2.
$$


\smallskip
In this section we treat $T$ as an {\bf ordered}\ tuple of {\bf ordered}\ pairs
(so that pairs and their elements can be unambiguously indexed).
To emphasize this notationally, we will from now on use parentheses as delimiters for pairs: $(x,y)$,
and square brackets as delimiters for pairings: $[(x_0,y_0), (x_1,y_1), (x_2,y_2)]$ etc.
Index values will begin at 0 or 1 as convenient. An abbreviated notation will be used,
where the above pairing will be written as $[(x_i,y_i)]_{i=0}^2$ or even as $[(x_i,y_i)]$ 
if the index range is clear from the context.

\smallskip
Allowing a slight abuse of notation, we will denote by $\Sigma_p$ two closely related but not altogether identical objects: a tuple (linearly arranged array) consisting of $3q+1$ ordered pairs and a corresponding table where the same pairs are arranged in a two-dimensional layout.

The tuple $\Sigma_p$ will be called the {\em extension}\ of the base starter $T$
and the table $\Sigma_p$  will be called the {\em  triplication table}.

\medskip
\noindent
\underline{Construction}.
 %
        Given a tuple 
        $$
        T=[(x_1,y_1), (x_2,y_2),\dots, (x_q,y_q)],
        $$
	we first form the auxilliary tuples 
    $$
    T_t^+=[(t+x_i,\ t+y_i)]\mmod p
    $$ 
    and 
    $$
    T_t^-=[(t-y_i,\ t-x_i)]\mmod p.
    $$

In order to obtain the {\bf table} $\Sigma_p$, we place the tuples $T,\ T_t^+,\ T_t^-$ in the first, second and third columns respectively.  
    The pair $(t,t)$ is placed on the top of the 2nd column the table. 
    The result is shown below.

    \bigskip
    
    \centerline{
	\begin{tabular}{|cc||cc||cc|}
		\hline 
		&& $t$, & $t$ & &  \\
		\hline
		$x_1$, & $y_1$ & $t+x_1$, & $t+y_1$ & $t-y_1$, & $t-x_1$\\ 
		\hline
		$x_2$, & $y_2$ & $t+x_2$, & $t+y_2$ & $t-y_2$, & $t-x_2$\\ 
		\hline
		$x_3$, & $y_3$ & $t+x_3$, & $t+y_3$ & $t-y_3$, & $t-x_3$\\ 
		\hline
		...,&...&...,&...&...,&...\\
		\hline
		$x_{q}$, & $y_{q}$ & $t+x_{q}$, & $t+y_{q}$ & $t-y_{q}$, & $t-x_{q}$\\ 
		\hline
	\end{tabular}
    }

\bigskip
The top row of the table consists of just one pair $(t,t)$, to be called the {\em top pair}. All other rows contain three pairs each and will be referred to as {\it regular rows}\ of the table $\Sigma_p$. 

   The {\bf extension tuple}\ 
   $$
   \Sigma_p =[(u_i,v_i)]_{i=0}^{3q}
   $$
   of length $3q+1$
   is obtained by reading the elements of the {\bf table} $\Sigma_p$
   row by row, left to right. For example, if $p=7$, then $q=3$, the extension tuple is 
   $$
   \Sigma_7=[(u_0,v_0),(u_1,v_1),\dots,(u_9,v_9)],
   $$
   where $u_i, v_i$ are the elements of the table.
   Let us exhibit the table in this new notation, indicating the linear index of each nonempty cell
   in parentheses:
%

	\bigskip
	
    \centerline{
	\begin{tabular}{|cc||cc||cc|}
		\hline 
		&&$(0)\quad u_0,$ & $v_0$ & &  \\
		\hline
		$(1)\quad u_1$, & $v_1$  & $(2)\quad u_2$, & $v_2$& $(3)\quad u_3$, & $v_3$\\ 
		\hline
		$(4)\quad u_4$, & $v_4$ & $(5)\quad u_5$, & $v_6$& $(6)\quad u_6$, & $v_6$ \\ 
		\hline
		$(7)\quad u_7$, & $ v_7$ & $(8)\quad u_8$, & $v_8$& $(9)\quad u_9$, & $v_9$ \\ 
		\hline
	\end{tabular}
    }

    \bigskip

The first formal statement concerning properties of $\Sigma_p$ is a generalization of exercise (a2)
of Section~\ref{sec:example}.
Consider the differences 
$$
\delta_i=u_i-v_i \mmod p.
$$
Observe that $\delta_0=t-t=0$. 
    
\begin{lemma}[Row property of the table $\Sigma_p$]
\label{lem:rowdifs}
    The pairs in the same regular row of the table $\Sigma_p$ have the same difference. 
    The set of these common differences is the same $\pmod p$ as the set of pair differences in the tuple $T$.
\end{lemma}

\begin{proof}
The $i$-th row ($1\leq i\leq q$) of the table $\Sigma_p$
consists of the pairs
$$
(u_{3i-2}, v_{3i-2})= (x_i, y_i),
\;
(u_{3i-1}, v_{3i-1})= (t+x_i, t+y_i),
\;
(u_{3i}, v_{3i})= (t-y_i, t-x_i).
$$
The differences $\delta_j$, $3i-2\leq j\leq 3i$, have the common value $d_i=x_i-y_i$.
\end{proof}


\noindent{\bf Corollary}
{\it If the tuple $T$ used in Construction is a  starter, then distinct regular rows of $\Sigma_p$ have distinct, non-zero differences $\pmod p$.}

\bigskip
\noindent
{\bf Example}.
Take the starter $T=\{\{2,\ 3\},\{4,\ 6\},\{1,\ 5\}\}$ of order 7 (the same as in Sec.~\ref{sec:example}). 
Take $t=1$.%
\footnote{This is an arbitrary choice. The construction can be carried out for any $t$, but the
solvability of the obtained Sudoku depends on $t$ --- see Theorem~\ref{thm:cond-key} of Sec.~\ref{sec:theorems}.}
%
The triplication table $\Sigma_7$ is as follows: 	

   \bigskip

\centerline{
	\begin{tabular}{|cc||cc||cc|}
		\hline 
		&&$(0)\quad 1,$ & $1$ & &  \\
		\hline
		$(1)\quad 2$, & $3$  & $(2)\quad 3$, & $4$& $(3)\quad 5$, & $6$\\ 
		\hline
		$(4)\quad 4$, & $6$ & $(5)\quad 5$, & $0$& $(6)\quad 2$, & $4$ \\ 
		\hline
		$(7)\quad 1$, & $ 5$ & $(8)\quad2$, & $6$& $(9)\quad 3$, & $0$ \\ 
		\hline
	\end{tabular}
}

\bigskip

Here $\delta_0=0$, $\,\delta_1=\delta_2=\delta_3=6$, $\,\delta_4=\delta_5=\delta_6=5$, $\,\delta_7=\delta_8=\delta_9=3$ $\pmod 7$.

This example will be continued and appended in the next two sections.

\section{Weak sets}    
\label{sec:weak}	

In this section we deal with pairs' sums modulo $p$ in $\Sigma_p$ (treated as a triplication table or as an extension of starter $T$)
$$
\sigma_i=u_i+v_i \mmod p, \quad  i=0,\dots,3q.
$$

\begin{definition}   
Let $W_s$ be the set of all pairs $\{(u_{\ell},v_{\ell})\}$ from $\Sigma_p$ with sum $\sigma_{\ell}=s$. The set $W_s$ is called  {\em weak set} in $\Sigma_p$ if either $s=0$ or $|W_s|>1$.
\end{definition}	

It follows from the definition that different weak sets are disjoint.
A pair $(u_i,v_i)$ is a {\em weak pair}\ if it belongs to one of the weak sets.  
A pair that is not weak is called {\em strong}.

The terminology is motivated by the fact that if $\Sigma_p$ (as a tuple of pairs) is the reduction modulo $p$
of a strong starter $S$ of order $3p$, then strong pairs in $\Sigma_p$, and only they,
have unique nonzero sums $\mmod p$. 

We say that $W$ is a weak set of type $r$ if $|W|=r$.

A weak set of type 1 (if it exists) consists of a single pair with sum $0\mmod p$.

\begin{lemma}[Weak sets are small.]
\label{lem:weak23}
If $T$ is a strong starter then for all weak sets $W$ in $\Sigma_p$ obtained by Construction, $|W|\le3.$
\end{lemma}

\begin{proof}
We want to show that there are no weak sets of type $>3$.
A weak set $W$ of type $>3$ must contain two different pairs in the same column of the table $\Sigma_p$.
It is easy to see that this cannot happen. 

In column 1:
if $(x_i,y_i)$ and $(x_j,y_j)$ are in $W$, then $\sigma_i=\sigma_j$ --- impossible since $T$ is strong.

Coincidences of pair sums in the regular rows of columns 2 and 3 are excluded similarly.

It remains to consider the case when one of the pairs in $W$ is the top pair $(t,t)$ with sum $2t$.
Suppose there exists a pair from $T_t^+$ (that is, from the second column and a regular row)
with the same sum $2t \mmod p$. We have
$(x_i+t)+(y_i+t)=2t\mmod p$. Hence $x_i+y_i=0\mmod p$, ---
 but this is forbidden by the definition of a strong starter.
\end{proof}

\begin{lemma}[Two zero sums at most]
\label{lem:weak1}
If $T$ is a strong starter, then in $\Sigma_p$ obtained by Construction
there are no more than two weak pairs with sum $0\mmod p$ (0-weak pairs): $|W_0|\le 2$.
\end{lemma}

\begin{proof}
The pair sums in  the first column, $T$, of the table $\Sigma_p$ are nonzero by the definition of a strong starter.

If there is a 0-weak pair in a regular row $i$ of the second column,
then $0=\sigma_{3i-1}=\sigma_{3i-2}+2t\mmod p$. Hence $\sigma_{3i-2}=-2t$. 

If there is a 0-weak pair in a regular row $j$ of the third column,
then $0=\sigma_{3j}=2t-\sigma_{3j-2}\mmod p$. Hence $\sigma_{3j-2}=2t$. 

Since all pair sums in $T$ are nonzero and distinct, in the case $t\neq 0$ there is at most one
0-weak pair in the second column and at most one in the third column. 

If $t=0$, then the only 0-weak pair in the table is the top pair $(0,0)$.
\end{proof}

\noindent
{\bf Example} (continued) 
    The following table shows weak sets in the triplication table given at the end of Section~\ref{sec:sigma_p}.
	
	\bigskip

    \centerline{
	\begin{tabular}{|l|c|c|c|c|}
		\hline
		Cell indexes & 2 & 4, 9 & 1, 5& 6, 7\\
		\hline
		Weak set & $\{(3,4)\}$ & $\{(4,6),\ (3,0)\}$ & $\{(2,3),\ (5,0)\}$ & $\{(2,4),\ (1,5)\}$ \\
		\hline
		Sum $\mmod 7$ & 0 & 3 & 5 & 6\\
		\hline
		Type & 1 & 2 & 2 & 2\\
		\hline
	\end{tabular}
 }	

\section{Monochrome sets}    
\label{sec:color}

In this section we work not with pairs from $\Sigma_p$ but with individual
numerical entries. We think of their values as of colors.

In Section~\ref{sec:sigma_p} we introduced the notation $u_i$, $v_i$ for the entries of tuple $\Sigma_p$.
It will be convenient to indicate the position of each entry by two numbers written in the angle brackets: $\ab{i,\ell}$, where $0\le i\le q$ is the index of the pair in $\Sigma_p$ and $\ell =0,1$ is the position of the entry within the pair:
$$
\ab{i,0}\;\; {\rm for}\,\, u_i, \qquad \ab{i,1}\;\; {\rm for} \,\,v_i,\quad 0\le i\le q.
$$

\begin{definition}
For $c\in\{0,1,\dots,p-1\}$, the {\em monochrome set}\ of color $c$ is the set of positions
$\ab{i,\ell}$  of entries with value $c$ in $\Sigma_p$. We denote this set by $M_c$.

\end{definition}

\begin{lemma}
\label{lem:colorsets}
The monochrome set of color $c\neq 0$ has cardinality $3$. The monochrome set of color 0
has cardinality $2$. That is, $$|M_c|=\begin {cases} 3\;,c\ne 0,\\ 2,\;c=0.\end{cases}$$
\end{lemma}

 Proof easily follows from Construction; the cases $t=0$ and $t\neq 0$ need to be treated separately.

\medskip
This Lemma shows that most monochrome sets are triples. The two-element
set of color $0$ is a ``degenerate triple'', but later we will artificially complete it to a full triple.

\bigskip
\noindent
{\bf Example} (continued) 
    The following table lists all monochrome triples in the triplication table given at the end of Section~\ref{sec:sigma_p}.
    For instance, the value 0 appears in pairs 5 and 9, both times in the second position  ($v_5=v_9=0$), which yields the set of color zero: $\{\ab{5,1},\,\ab{9,1}\,\}$. 
	
	\bigskip
	
\centerline{
    \begin{tabular}{|c|c|c|}
		\hline
		color & monochrome sets of positions $\ab{i,\ell}$& variables\\
		\hline
		0& $\ab{5,1}$, $\ab{9,1}$ & $v_5, v_9$\\
		\hline
		1& $\ab{0,0}$, $\ab{0,1}$, $\ab{7,0}$ & $u_0, v_0, u_7$\\
		\hline
		2& $\ab{1,0}$, $\ab{6,0}$, $\ab{8,0}$  &$u_1, u_6, u_8$\\
		\hline
		3& $\ab{1,1}$, $\ab{2,0}$, $\ab{9,0}$ &$v_1, u_2, u_9$\\
		\hline
		4& $\ab{2,1}$, $\ab{4,0}$, $\ab{6,1}$ &   $v_2, u_4, v_6$\\
		\hline
		5& $\ab{3,0}$, $\ab{5,0}$, $\ab{7,1}$ & $u_3, u_5, v_7$\\
		\hline
		6& $\ab{3,1}$, $\ab{4,1}$, $\ab{8,1}$ & $v_3, v_4, v_8$\\
		\hline
	\end{tabular}
}

\section{Setting up Sudoku-type problem modulo 3} 
\label{sec:sigma_3}

The tuple of pairs $\Sigma_p=\{(u_i,v_i)\}_{i=0}^{3q}$ is determined by Construction described in
Section~\ref{sec:sigma_p}. 
For given numerical dataset (the base starter $T$ and key $t$), its elements are integer numbers modulo $p$.

We introduce a tuple of pairs $\Sigma_3=\{(U_i,V_i)\}_{i=0}^{3q}$. Its elements are thus far 
indeterminates. Our task will be to find their numerical values as {\bf integers modulo 3}.

Hence the initial, by-definition constraint:
\begin{enumerate}
\item[(0)]
{\bf Range constraints}.

 All variables $U_i, V_i$ take values in the set $\{0,1,2\}$.
\end{enumerate}

Let us arrange the set of unknowns $\{U_i,V_i\}$ in a table format as in the triplication table $\Sigma_p$. The so obtained arrangement will be referred to as the {\em table $\Sigma_3$}.

\medskip
Three classes of constraints are imposed on the entries of the table $\Sigma_3$:

\begin{enumerate}
\item[(1)]
{\bf Row Constraints}.
We require that the differences $D_i=U_i-V_i$  in each regular row of the table be distinct $\mod 3$, that is,
$$
\{D_{3i-2}, D_{3i-1}, D_{3i}\}=\{0,1,2\}, \quad i=1,2,\dots,q.
$$

\item[(2)]
{\bf Weak Set Constraints}.
Every weak set of the triplication table $\Sigma_p$ gives rise
to one constraint.

Let $\{(u_i,v_i)\}$ be the set of pairs (with the common value of the sums $\sigma_i=s$) comprising the given weak set $W_s$. We denote the corresponding
set of indexes $i$ by $\ind(W_s)$. 

The form of the constraint is slightly different for $W_s$ with $s=0$ and $s\ne 0$.

\begin{enumerate}

\item[(a)] In the case $s\neq 0$, it is required that the sums  
$\{U_i+V_i,\, i\in\ind(W_s)\}$, be distinct. Since there are 3 distinct values modulo 3, this requirement by itself
is not forbidding. A similar requirements involving a set of 
cardinality greater than 3
would be forbidding, but there is no such a weak set (Lemma~\ref{lem:weak23}).

\item[(b)] In the case $s=0$, it is required that the sums $\{U_i+V_i,\ i\in \ind(W_0)\}$ be distinct and nonzero. This requirement by itself is not forbidding since there are two distinct nonzero values,
 $1$ and $2$  modulo 3 and $|W_0|\leq 2$ by Lemma~\ref{lem:weak1}.
 
\end{enumerate}

\item[(3)]
{\bf Color Triple Constraints}.
We augment the two-element monochrome set $M_0$ of color 0 by a dummy double index $\ab{*,*}$.
We append a dummy variable $Z=0$ to the set of unknowns $\{U_i,V_i\}$ and declare
that the index (position) of $Z$ in the table $\Sigma_3$ is $\ab{*,*}$.

To each color $0\le c\le p-1$ there corresponds one constraint: the values $\pmod 3$ of the variables in positions $\ab{i,\ell}\in M_c$
must be distinct. 

The constraint corresponding to color 0, formulated without dummies,
says that the two variables  in positions $\ab{i,\ell}\in M_0$ 
must take
distinct values $1$ and $2$.

\end{enumerate}

\section{Theoretical results}
\label{sec:theorems}

\begin{theorem}[Solution of Sudoku yields a strong starter]
\label{thm:sudoku_to_starter}
Let $\Sigma_p=\{u_i,v_i\}_{i=0}^{3q}$ be a tuple obtained by Construction from a strong starter $T$ of order $p=2q+1$. 
 Suppose $\{U_i,V_i\}_{i=0}^{3q}$ is the set of values $\mmod 3$
 satisfying the constraints {\rm (1)--(3)} 
 of Section~\ref{sec:sigma_3}
 and associated with $\Sigma_p$.   

For $i=0,\dots,3q$, let $a_i$, $b_i$ be the values $\mod 3p$ satisfying the congruences
$a_i\equiv u_i\mmod p$, $a_i\equiv U_i\mmod 3$, $b_i\equiv v_i\mmod p$, $b_i\equiv V_i\mmod 3$.
(The solutions of the congruences are determined by CRT.)
 Then the set of pairs $S=\{\{a_i,b_i\}\}_{i=0}^{3q}$ is a strong starter in $\ZZ_{3p}$.
\end{theorem}

\begin{proof}
The constructed pairing $S$ has $3q+1$ pairs. Since the starter $T$ is of order $p=2q+1$,   
the order of $S$ is $n=2(3q+1)+1=6p+3=3(2q+1)=3p$. 

\smallskip
We have to check that:

\begin{enumerate}
    
\item[(a)] $S$ is a partition of the set $\ZZ_{3p}^*$;

\item[(b)] The set of differences of pairs in $S$, 
$\{\pm(a_i-b_i)\mmod 3p\}_{i=0}^{3q}$, coincides with $\ZZ_{3p}^*$.

\item[(c)] The pair sums $\{a_i+b_i\}_{i=0}^{3q}$ are all nonzero and distinct.
\end{enumerate}

Property (a) follows from Color Constraints: any two values from the list $a_0,b_0,\dots,a_{3q},b_{3q}$
either are different modulo $p$ or, if not, --- they are different modulo 3. 

\smallskip
Property (b) is ensured by Row Constraints with a little help from Color Constraints. 
For pairs in regular rows, we have
$a_i-b_i\equiv d_i\mmod p$, $a_i-b_i\equiv D_i\mmod 3$. By our set-up, since $T$ is a starter,
the pairs $(d_i,D_i)$ are all distinct, $d_i\not\equiv 0$, and $d_i\not\equiv -d_j \mmod p $ for all $i,j$.
It remains to look at the difference $a_0-b_0$. We have $a_0-b_0\equiv t-t=0\mmod p$, $a_0-b_0\equiv D_0
\neq 0\mmod 3$, since the positions $\ab{0,0}$ and $\ab{0,1}$ belong to the monochrome triple of color $t$.

\smallskip
Property (c) follows from Weak Set Constraints. They imply that the pairs
$(\,(a_i+b_i)\mmod p,\, (a_i+b_i)\mmod 3\,)$ are all distinct, hence the sums $(a_i+b_i)\mmod 3p$
are all distinct.

It remains to check that there are no zero sums $\mmod 3p$. 
But if $a_i+b_i\equiv 0\mmod p$, then the pair $(u_i,v_i)$ in $\Sigma_p$ is weak.
By item (b) in the description of Weak Set Constraints, the corresponding sum $(U_i+V_i)\mmod 3\in\{1,2\}$,
hence $(a_i+b_i)\mmod 3p\neq 0$.
\end{proof}

\begin{theorem}[Necessary condition for Sudoku solvability] 
\label{thm:cond-key}
If the value of key $t$ in Construction of Sec.~\ref{sec:sigma_p} is either 0 or one of the pair
sums $x_i+y_i$ of the base starter $T$ (not necessarily strong), then the modular Sudoku problem does not have a solution.
In other words, the condition $t\notin\hat T\cup\{0\}$ is necessary for solvability of the
modular Sudoku problem.
\end{theorem}

\begin{proof}
In the case $t=0$, we have two equal pairs in every regular row of the triplication table:
$(u_{3i-2},v_{3i-2})=(u_{3i-1},v_{3i-1})$ (both pairs coincide with pair $(x_i,y_i)$ of the base starter $T$).

In the case where $(x_i+y_i)\mmod p = t$ for some $i$, we again have two equal pairs in one of the rows ($i$-th)
of the triplication table: $(u_{3i-2},v_{3i-2})=(u_{3i},v_{3i})$.
Indeed, calculating in $\ZZ_p$, we get $u_{3i}=t-y_i=(x_i+y_i)-y_i=x_i=u_{3i-2}$ and similarly $v_{3i}=y_i=v_{3i-2}$. 

It remains to prove that the existence of two identical pairs $(u,v)=(u',v')$
in the same row of the table $\Sigma_p$
is forbidding for the existence of solution of the modular Sudoku problem.

The elements of pairs $(U,V)$ and $(U',V')$ of the table $\Sigma_3$ must satisfy the constraints:
\\
(1) $U-V\not\equiv U'-V' \mmod 3$ (part of Row Constraints);
\\
(2) $U+V\not\equiv U'+V'\mmod 3$ (part of Weak Set Constraints);
\\
(3) $U\neq U'$, $V\neq V'$ (by Color Constraints: the positions of the elements $U$, $U'$ belong to the same monochrome set of  color $u$
and the positions of the elements $V$, $V'$ belong to the same monochrome set of  color $v$).

The constraints (1) and (2) imply: $U-U'\not\equiv \pm(V-V')\mmod 3$. 
Thus, if $V-V'\not\equiv 0\mmod 3$, then $(U-U')\mmod 3\notin\{1,2\}$, hence $(U-U')\mmod 3=0$.

We see that either $U=U'$ or $V=V'$. Either case violates one of the constraints (3). 
\end{proof}

\noindent
In all our test examples (see Sec.~\ref{sec:computations}) where the necessary condition of Theorem~\ref{thm:cond-key} was met,
the modular Sudoku problem did have a solution. 

Therefore we propose

\medskip
\noindent
{\bf Conjecture}. {\em If $T$ is a strong starter of odd order $p\geq 7$ coprime with 3 and $t\in\ZZ_n^*\setminus \hat{T}$, where $\hat T$ is the set of pair sums of $T$, then the modular Sudoku problem corresponding to the data $(T,t)$ has a solution.}

\begin{theorem}[Solutions of Sudoku come in pairs] 
\label{thm:sudoku-symmetry}
Let $\phi$ be a permutation of the set $\{0,1,2\}$ that transposes $1$ and $2$. 
If the set of values $\{U_i,V_i\}$ satisfies the constraints of the modular Sudoku problem,
then the set of values $\{\phi(U_i),\phi(V_i)\,\}$ also satisfies those constraints.
\end{theorem}

\begin{proof}
It is straightforward to see that every constraint is invariant under transposition of values $1$ and $2$.
\end{proof}

\noindent{\bf Remark}.
The permutation $\phi$ is an automorphism of the group $\ZZ_3$: $\phi(X)=-X\mmod 3$. 
Moreover, the map $(a\mmod p, \;A\mmod 3)\mapsto (a\mmod p, \;\phi(A\mmod 3))$ is an automorphism of the group
$\ZZ_{3p}\cong\ZZ_p\oplus\ZZ_3$ that fixes the direct summand $\ZZ_p$.
From this point of view, Theorem~\ref{thm:sudoku-symmetry} is a simple consequence of a well-known partition of starters into equivalent classes (see e.g.\ Introduction in \cite{Vesa10}). 

\medskip

Let us briefly address the question mentioned in Section~\ref{sec:example}, item 8.

\medskip\noindent
{\bf Question.} (Inverse problem of triplication.) 
{\em Given a strong starter $S=\{\{a_i,b_i\}\}_{i=0}^{3q}$, determine whether $S$ can be obtained from 
some pair $(T,t)$, where $T$ is a strong starter of order $p=2q+1$, $t\in\ZZ^*_{p}\setminus \hat T$, by triplication.}

\medskip

We offer a partial answer --- a sufficient condition that $S$ is not a result of triplication.
The condition is based solely on the row structure of the table $\Sigma_p$.%
\footnote{To be appreciated by Cindy of exercise (d), Sec.~\ref{sec:example}.}%

It is convenient to formulate the condition in the form of an algorithm with two return values: {\em False} (meaning that $S$ is not a result of triplication) and {\em Inconclusive}.

\smallskip
Step 1. Reduce the given starter $S$ modulo $p$; denote the result by $S\mmod p$.

\smallskip
Step 2. For each $d\in\{0,\dots, q\}$ identify the set $R_d$ (``row with difference $d$'')
of {\em ordered}\ pairs $(a,b)$ such that $\{a,b\}\in S\mmod p$ and $(a-b)\mmod p = d$.
Since $S$ is a starter, we have $|R_0|=1$ and $|R_d|=3$ for $d\neq 0$.

\smallskip
Step 3. $R_0=\{(t,t)\}$ is a singleton, call $t$ the ``key'' and continue. 

\smallskip
Step 4. For every set $R_d$, $d\neq 0$, check 
whether 
a certain ordering of $R_d$
is a valid row of a triplication table, that is, has the form
$$
(u,v), \; (u',v'), \; (u'',v'')
$$
with
$$
u'+v''\equiv v'+u''\equiv 2t\pmod p, \qquad u'-u \equiv v'-v \equiv t\pmod p.
$$
For that purpose, examine all 6 permutations of the pairs in $R_d$. 
If no permutation satisfies the stated property, stop and return {\em False}. 

\smallskip
Step 5. If all rows $R_d$, $d=1,\dots,q$, pass the test of Step 4, return {\em Inconclusive}.

\medskip
The proposed test is simple yet practical; see Discussion at the end of Sec.~\ref{sec:computations}. 

\medskip\noindent
{\bf Example 1}. Consider the starter $S$ of order 21:
  $$
  \{\{13, 12\}, \{19, 17\}, \{7, 4\}, \{10, 14\}, \{15, 20\}, \{3, 9\}, \{1, 8\}, \{5,18\}, \{11, 2\}, \{16, 6\}\}.
  $$

The sets $R_d$ modulo 7 relevant for our purposes are:
$R_0=\{(1,1)\}$ (comes from pair $\{1,8\}$ and determines the key value $t=1$) 
and $R_3=\{(0,4),(3,0),(2,6)\}$
(comes from pairs $\{7,4\}$, $\{10,14\}$, $\{16,6\}$).

The set $R_3$ fails the test of Step 4. Hence $S$ is not obtainable by triplication.

\medskip\noindent
{\bf Example 2}. Here we illustrate a possibility that $S$ can be obtained by triplication from a starter but not from a strong starter.

Consider the strong starter $S$ of order 39:
$$
\begin{array}{c}    
\{\{2, 1\}, \{36, 34\}, \{6, 3\}, \{8, 12\}, \{38, 33\}, \{16, 10\}, \{18, 11\}, \{29, 21\}, 
\{22, 31\},  \{25,
\\[0.5ex]
 15\},\{13, 24\}, \{20, 32\}, \{30, 17\}, \{23, 37\}, \{19, 4\}, \{28, 5\},
\{26, 9\}, \{14, 35\}, \{7, 27\}\}
\end{array}
$$


The starter $S$ passes the test. 
We leave it to the reader to see that 
the only pairing $T$ of order 13 that can be triplicated to $S$ (with key $t=4$) is
$$
T=\{(11,10),(6,4),(2,12),(9,5),(8,3),(7,1)\}.
$$
It is a starter but not a strong starter.

\smallskip
The latter starter $T$ also gives a counterexample to a strengthened version of the above stated Conjecture,
where the requirement that the base starter be strong is dropped. In this case, $3\notin \hat T$, yet
the modular Sudoku problem obtained from the data $(T,\,t=3)$ does not have a solution.
To see it, note that there are 4 pairs with sum $\equiv 1\pmod{13}$ in the triplication table $\Sigma_{13}$.
The existence of  a Sudoku solution and therefore a corresponding strong starter  
contradicts Lemma~\ref{lem:weak23}.


\section{Computations and discussion}
\label{sec:computations}

\subsubsection*{About the program}

We implemented the described method of starter triplication in Python (Python 3.11.5)
using SAT solver library \zsat (z3-solver 4.15.0.0).

Programming the constraints defining the modular Sudoku problem is straightforward with functions
of \zsat library.  
For clarity of code, along with unknowns of interest, $U_i, V_i$ ($i=0,\dots,3q$),
the auxiliary unknowns  are used in the code: 

\begin{itemize}
\item
``differences'' $D_i$ ($i=0,\dots,3q$)
with constraints $0\leq D_i\leq 2$, $D_i=U_i-V_i\pmod 3$,
\item
``weak sums'' $S_i$ ($i\in W=\cup \ind(W_s)$, where $W_s$ are the weak sets)
with constraints $0\leq S_i\leq 2$, $\;S_i=U_i+V_i\pmod 3$,
\item
the ``dummy variable'' $Z$ with fixed value $Z=0$.
\end{itemize}

Below we explain in detail how the \zsat solver is programmatically set up in the special case of Example considered in Sections~\ref{sec:sigma_p}--\ref{sec:color} (``Demo Example''). The full listing of a working Python program for this example is given in Appendix.

In general, the base strong starter and the value of the key are variable parameters;  the computation of weak sets 
(Sec.~\ref{sec:weak}) and monochrome sets (Sec.~\ref{sec:color}) should be automated.  This has been done --- for computations reported later in this section. 
Since that part of code is independent of any particular method or algorithm for solving the modular Sudoku problem, 
it is not presented here.

Another solver-independent part is a utility that merges the triplication table $\Sigma_p$ and the found solution
modulo 3 in table $\Sigma_3$ into a starter of order $3p$ using CRT. 
For that purpose, in the demo program in Appendix, a programming trick using Python's `dictionary' data structure is employed. For instance, the ``label'' $(5,1)$ in the dictionary points to the value $19$, which is congruent to $5\mmod 7$ and to $1\mmod 3$. 
A second starter (see Theorem~\ref{thm:sudoku-symmetry}) is recovered similarly from the same solution of Sudoku.

\subsubsection*{Setting up the solver in Demo Example}

We refer to Example followed through in Sections~\ref{sec:sigma_p}--\ref{sec:color}.

Recall that the order of the base starter is $p=7$, of the resulting starter $3p=21$;
the parameter $q=(p-1)/2=3$.

The unknowns are: $U_i, V_i, D_i$ ($i=0,\dots,9$), $S_i$ ($i\in W$, $W=\{1,2,4,5,6,7,9\}$ --- set of indexes of weak pairs), and $Z$. 

Here is the complete list of constraints defining the Sudoku-type problem mod 3 as pertain to the data in Example. Instead of the mathematical inequality sign ``$\neq$'', we use the keyword {\tt Distinct}\ defined
in the \zsat library.

\begin{enumerate}
    \item[0.0)] Dummy variable:
       $Z=0$
    
    \item[0.1)] Binding main unknowns and differences: 	
    $U_i-V_i-D_i=0 \mmod 3\;$ for $\;0\le i\le 9$.

    \item[0.2)] Binding main unknowns and sums:	
	$U_i+V_i-S_i=0 \mmod 3\;$ for $\;i\in W=\{1,2,4,5,6,7,9\}$.

    \item[1)] Range Constraints: 	
	$0\leq U_i\leq 2$, $\quad 0\leq V_i\leq 2$, $\quad 0\leq D_i\leq 2$\; for $\;0\le i\le 9$,
    $0\leq S_i\leq 2$ for $i\in W$.
	
    \item[2)] 	
        Row Constraints: \\
	\qquad {\tt Distinct}($D_{3i-2}, D_{3i-1}, D_{3i}$)\; for $\;1\le i\le 3$.
	
    \item[3.1)] 	
	Weak Set Constraint for weak pair with sum 0: \\
    \qquad $S_2\ne 0 \mmod 3$ or, equivalently:
    \qquad {\tt Distinct}$(S_2, Z)$. 

    \item[3.2)] 	
	Weak Set Constraints involving weak pairs with nonzero sums:  \\
    \qquad {\tt Distinct}$(S_i, S_j)$ 
     for $(i,j)$ in $\{(4,9),(1,5),(6,7)\}$.

   \item[4)] Color Constraints:
\\  \qquad {\tt Distinct}$(V_5,V_9,Z)$ (for degenerate monochrome set $M_0$),
\\  \qquad 
{\tt Distinct}$(U_0, V_0, U_7)$,	
\quad{\tt Distinct}$(U_1,U_6,U_8)$,	
\quad{\tt Distinct}$(V_1,U_2, U_9)$,	
\\  \qquad 	
{\tt Distinct}$(V_2,U_4, V_6)$,
\quad{\tt Distinct}$(U_3,U_5, V_7)$,	
\quad{\tt Distinct}$(V_3, V_4, V_8)$.
\end{enumerate}

\subsubsection*{Technical details}

Computations were carried out in the Jupiter Notebook environment on a fanless mini PC with 11th Gen Intel(R) Core(TM) i7-1165G7  2.80GHz
processor and 32 GB memory under Windows 11 operating system.  
The task was run on one of the four CPU kernels and, according to occasional checks with Task Manager,
occupied between 90\% and 100\% of that kernel's load.
Certain issues beyond our control and observation could potentially affect relative performance at different
stages of computation --- such as system's self-maintenance and the CPU heat control.

Our judgement about solver's performance in different series of experiments relies on recorded elapsed times.
Time benchmarks, of course, depend on system's particulars.
Possibly, benchmarks such as the number of solver's cycles might give more objective and invariant information about the behavior of the algorithm.
But the solver's inner statistics was not available to us.   

\subsubsection*{Specification of numerical experiments}

Base starters for the numerical experiments were generated in advance using a hill-climbing algorithm of Dinitz and Stinson
\cite{DinitzStinson81, Stinson85}.

\smallskip
Performance results for three series of experiments are reported below.

\begin{itemize}
\item
Series 1: {\em Base starters of orders $<100$, all admissible key values}.

For odd orders $m$ from $7$ to $97$ coprime with $3$, a fixed strong starter $T_m$ of order $m$ was taken
as a base starter. For every key value satisfying the necessary condition $t\in\ZZ^*_{p}\setminus \hat T_m$
(Theorem~\ref{thm:cond-key}), the modular Sudoku problem was set up and solved.

\item
Series 2: {\em Base starters of larger orders, random admissible key values}.

The set of orders in this series is the union of the following:
\begin{itemize}
\item
all odd orders $m$ from $101$ to $199$ coprime with $3$;
\item all orders in arithmetical progression with difference $12$ from $203$ to $335$;
\item orders $395$, $439$, $499$.
\end{itemize}
For every order from this list, 
a random key value   $t\in\ZZ^*_{p}\setminus \hat T_m$
was chosen, the modular Sudoku problem was set up and solved.

\item
Series 3: {\em Fixed order $m=101$, all admissible key values, multiple runs}.

A fixed strong starter $T_{m}$ of order $m=101$ (same for the whole series) was taken as a base starter.
In a subseries, the modular Sudoku problem was solved for every key value $t\in\ZZ^*_{p}\setminus \hat T_m$.
The subseries were executed ten times.
\end{itemize}

\subsubsection*{Performance results}

The measure of performance is execution time in seconds, denoted by $\tau$. 

\medskip
\noindent{\bf Numerical experiments --- Series 1}.

\medskip
Average time (in seconds) per Sudoku is shown in Table \ref{tab:time_vs_order1}.
Here, $m$ is the order of a base strong starter $T_m$, averaging is done over all $(m-1)/2$
admissible values of key $t$.

\begin{table}[h]

\centerline{	
\begin{tabular}{|c|c||c|c||c|c||c|c|}
		\hline
$m$ & time & $\;m$ & time & $\;m$ & time & $\;m$ & time \\
		\hline
7 & 0.065 & 29 & 5.276 & 53 & 21.35 & 77 & 53.07 \\
9 &  --     & 31 & 5.216 & 55 & 24.92 & 79 & 52.23 \\
11 & 0.118 & 35 & 8.741 & 59 & 26.09 & 83 & 78.46 \\
13 & 0.309 & 37 & 8.396 & 61 & 33.90 & 85 & 70.72 \\
17 & 1.017 & 41 & 10.51 & 65 & 35.91 & 89 & 73.39 \\
19 & 1.362 & 43 & 11.96 & 67 & 40.50 & 91 & 78.64 \\
23 & 2.674 & 47 & 17.09 & 71 & 49.75 & 95 & 93.73 \\
25 & 3.817 & 49 & 20.55 & 73 & 48.98 & 97 & 90.14 \\
		\hline
	\end{tabular}
}	
\caption{Average execution time for solving Sudoku of orders $<100$} 
\label{tab:time_vs_order1}
\end{table}

For most values of $m$, a significant spread between the minimum and maximum values of time
was observed: generally, about 2 times. Excluding $m=7$, the maximum observed spread was
$3.5$ times for $m=19$ and the minimum observed ratio was $1.32$ for $m=37$.

The data are plotted in Fig. \ref{fig:time_vs_order1} together with heuristic interpolation of the
average time given by the
formula
\begin{equation*}
\label{heuint1}
\tau_1(m)=0.023\cdot m^2\,\ln m.
\end{equation*}

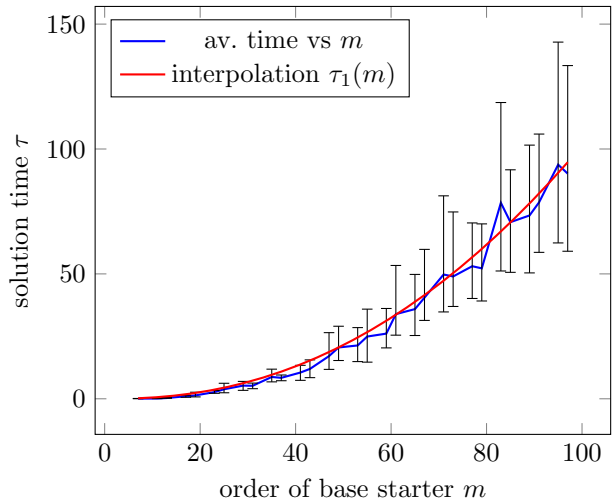
\begin{figure}[h]
\centerline{	
\begin{tikzpicture}
\pgfplotstableread{time_7-99.txt}\datatable
\begin{axis}[
  xlabel=order of base starter $m$,
  ylabel=solution time $\tau$,
  legend pos=north west,
     ]

\pgfplotstablecreatecol[
    create col/expr={\thisrow{avtime} - \thisrow{mintime}}
]{minus}{\datatable}
\pgfplotstablecreatecol[
    create col/expr={\thisrow{maxtime} - \thisrow{avtime}}
]{plus}{\datatable}

\addplot+[
    no marks,
    thick,
    error bars/.cd,
        y dir=both,
        y explicit,
        error bar style={black}, 
] table [
    x=m,
    y=avtime,
    y error minus=minus,
    y error plus=plus,
] {\datatable};

\addplot+[no marks, thick, domain=7:97]{x^2 * ln(x)* 0.0022};

\addlegendentry{av.\ time vs $m$}
\addlegendentry{interpolation $\tau_1(m)$}
\end{axis}
\end{tikzpicture}
}
\caption{Execution time for solving Sudoku of orders $<100$}
\label{fig:time_vs_order1}
\end{figure}

\medskip
\noindent{\bf Numerical experiments --- Series 2}.

\medskip
Time (in seconds) per Sudoku is shown in Table \ref{tab:time_vs_order2}.
Here, $m$ is the order of a base strong starter $T_m$. A single computation was done for every order.
For formatting reasons, the table does not include data for orders $m\in (100,200)$ with $m\mmod 6 =1$. 
However, those data are included in the plot shown in Fig.\ref{fig:time_vs_order2}.


\begin{table}[h]
\centerline{	
\begin{tabular}{|c|c||c|c||c|c||c|c|}
		\hline
$m$ & time & $\;m$ & time & $\;m$ & time & $\;m$ & time \\
		\hline
101 & 124 & 149 & 374 & 197 & 629 & 287 & 2267 \\
107 & 94 & 155 & 337 & 203 & 735 & 299 & 4785 \\
113 & 120 & 161 & 468 & 215 & 985 & 311 & 3464 \\
119 & 153 & 167 & 306 & 227 & 1586 & 323 & 3384 \\
125 & 253 & 173 & 485 & 239 & 1280 & 335 & 4737 \\
131 & 199 & 179 & 407 & 251 & 2750 & 395 & 17344 \\
137 & 232 & 185 & 670 & 263 & 2033 & 439 & 20788 \\
143 & 251 & 191 & 593 & 275 & 1879 & 499 & 17962 \\
		\hline
	\end{tabular}
}	
\caption{Execution time for solving Sudoku of orders between 100 and 500}
\label{tab:time_vs_order2}
\end{table}


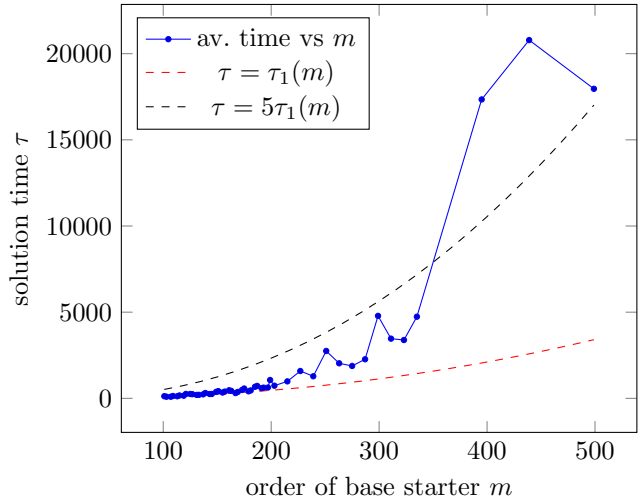
\begin{figure}[h]
\centerline{	
\begin{tikzpicture}
\pgfplotstableread{time_101-499.txt}\datatable
\begin{axis}[
  mark size = 1pt,
  xlabel=order of base starter $m$,
  ylabel=solution time $\tau$,
  legend pos=north west,
      x filter/.expression={x<100 ? nan : x},
      scaled ticks=false,
      y tick label style={
        /pgf/number format/.cd,
         fixed, set thousands separator={},
        /tikz/.cd
    },
     ]
\addplot table [y=time, x=m]{\datatable};

\addplot+[no marks, thin, dashed, domain=101:499]{x^2 * ln(x)* 0.0022};
\addplot[no marks,  thin, dashed, domain=101:499]{x^2 * ln(x)* 0.011};

\addlegendentry{av.\ time vs $m$}
\addlegendentry{ $\tau=\tau_1(m)$}
\addlegendentry{ $\tau=5\tau_1(m)$}
\end{axis}
\end{tikzpicture}
}	
\caption{Execution time for solving Sudoku of orders between 100 and 500}
\label{fig:time_vs_order2}
\end{figure}

We observe that for $m\lesssim 200$, time values fit rather well in the heuristic ``law'' $\tau\approx \tau_1(m)$
seen in the discussion of Series 1.

For larger values of $m$, that approximation is no longer good. For comparison, we plotted the curve (black)
$\tau=5\tau_1(m)$. 

We note that the observed time values do not change monotonically with order. 
There is a dramatic jump between the points with $m=335$ and $m=395$ having no analogs in
other explored intervals of the $m$ axis. 

A natural question arises: do the results agree with or testify against the hypothesis
about polynomial bound of time required (with high probability) for solving the modular Sudoku problem by a SAT solver?

Althought the hypothesis looks counterintuitive, we see no clear evidence against it in the obtained data.

\medskip
\noindent{\bf Numerical experiments --- Series 3}.

\medskip
The goal of this series was to explore variations of solution time $\tau$ between different runs of the solver with identical
input data $(T, t)$ and between different values of key $t$ (for the same starter $T$) averaged over
several runs.

For each of $50$ admissible (i.e.\ satisfying the necessary conditions of Theorem~\ref{thm:cond-key})
value of $t$, the same Sudoku problem was solved 10 times. We call a sequence of 50 trials
for different admissible $t$ (ran consecutively in real time) a subseries. 
The program was thus executed 500 times in total, split into 10 subseries. 

\smallskip
The extreme values of time $\tau$ (in seconds) occurred:
\begin{itemize}
\item
Near-minimum $\tau=82.98$ in subseries 2 for key $t=30$,

\item
Minimum $\tau=82.95$ in subseries 2 for key $t=62$,

\item
Maximum $\tau=260.7$ in subseries 7 for key $t=10$,

\item
Near-maximum $\tau=257.4$ in subseries 10 for key $t=10$.

\end{itemize}



No uniformity was observed in the distribution of execution times over different values of keys, $t$.
It appears that for some key values the triplication process completes systematically quicker than
for others. The histogram  shown in Fig. \ref{fig:histogram} illustrates this thesis.

For each key value, the execution times were averaged over 10 trials. The averages ranged from
106.9 sec for $t=42$ to 201.6 sec for $t=10$. The average times were grouped into ranges of length 10
and, for every such range, the number of keys with average times in that range was found.

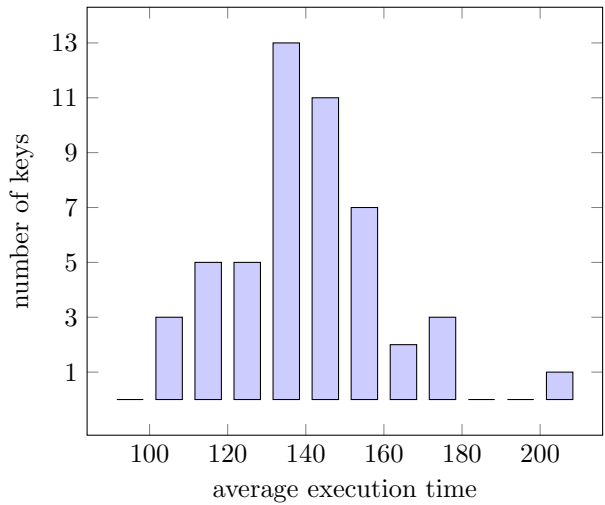
\begin{figure}[h]
\centerline{
\begin{tikzpicture}
        \begin{axis}
        [fill=blue!20,
         xtick={100,120,...,200},
        ytick={1,3,5,7,9,11,13},
	 ylabel = {number of keys},
        xlabel = {average execution time},
        ]
            \addplot[range=100:200,
             const plot, ybar,fill=blue!20]
             coordinates{(95,0)(105,3) (115,5) (125,5) (135,13)(145,11)(155,7)(165,2)(175,3)(185,0)(195,0)(205,1)};
       \end{axis}
 \end{tikzpicture}
 }
 \caption{Distribution of keys in the execution time intervals for a Sudoku of order 101}
 \label{fig:histogram}
\end{figure}

\bigskip

The final plot shown in Fig. \ref{fig:chaos} illustrates chaotic oscillations of the solution time around a trendline.

First, we order the keys so that to have average time (over 10 subseries) to grow monotonically with the 
index (with respect to the new ordering) of the key. 

After that, we plot solution time vs the above-defined index of the key for several subseries;
to avoid excessive cluttering, we selected the subseries with numbers 1, 4, 7, 10.
The dashed line is the ``trendline'' --- time averaged over all subseries for the given key.

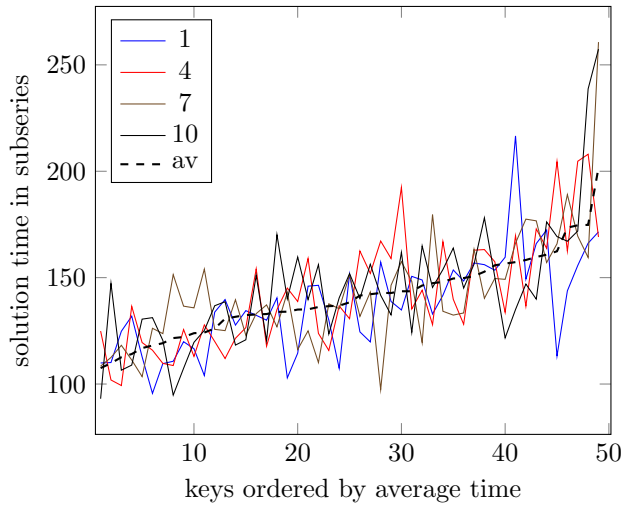
\begin{figure}[h]
\centerline{
\begin{tikzpicture}
\pgfplotstableread{ser3_keys_ordered_by_average_time.txt}\datatable
\begin{axis}[
    xmin=0.5,    
    xmax=50.2,    
  xlabel=keys ordered by average time,
  ylabel=solution time in subseries,
  legend pos=north west,     
]
\addplot+[thin,no marks,] table [x expr={\coordindex > .5 ? \coordindex : nan},
    y=1
] {\datatable};
\addlegendentry{1}

\addplot+[thin,no marks,] table [x expr={\coordindex > .5 ? \coordindex : nan},  
    y=4
] {\datatable};
\addlegendentry{4}

\addplot+[thin,no marks,] table [x expr={\coordindex > .5 ? \coordindex : nan}, 
    y=7
] {\datatable};
\addlegendentry{7}

\addplot+[thin,no marks,] table [x expr={\coordindex > .5 ? \coordindex : nan},  
    y=10
] {\datatable};
\addlegendentry{10}

\addplot+[thick, no marks, color=black, dashed,] table [x expr={\coordindex > .5 ? \coordindex : nan},  
    y=av
] {\datatable};
\addlegendentry{av}

\end{axis}
\end{tikzpicture}
}
\caption{Oscillations of the solution time
around a trendline in four subseries}
\label{fig:chaos}
\end{figure}
    
\subsubsection*{Discussion}
1. A SAT solver can be employed directly for finding strong starters. 
However, longer time is required for finding a strong starter even for small orders and the time grows
much faster as the order increases. 
	
For comparison, using \zsat solver directly, computation of a strong starter of order $n=39$ took about 200
seconds on the same system, while triplication of a starter of order $p=13$ required less that 1 second. 

\smallskip
2. On the other hand, generating strong starters by hill-climbing algorithm of Dinitz and Stinson is
much faster: generation of a sampling of 1000 strong starters of order 39 takes less than 2.5 sec.

\smallskip
3. We also note that we (OO) solved modular Sudoku problems of orders $3p$ by hand for 
$p$ up to 47.
The time per problem in the upper range of $p$ was about an hour. The employed approach was, in all likelihood, very different from the internal workings of a universal
SAT solver. We expect to formalize the process; a specialized solution algorithm
is expected to significantly outperform the universal solver; it is to be compared
with hill-climbing approach.

\smallskip
4. Our final remark is unrelated to efficiency of solvers but concerns properties of 
strong starters that distinguish ``triplicated'' starters in the totality of all strong starters 
of the given order $3p$. 
We report a small, very preliminary, yet perhaps an indicative result in this direction.
We sampled%
\footnote{Repetititions were allowed and, for order 21, inevitable: there are only $6660$ different
strong starters of order 21 \cite{Vesa10}.}
(by hill-climbing) 
 $10^6$ strong starters of each of the orders $21$, $33$ and $39$
and applied the criterion described at the end of Sec.~\ref{sec:theorems} to examine them.
Only a small fraction of those ``random'' strong starters of order 21 were found ``inconclusive''.
All others failed 
the test.
Specifically, out of $10^6$ generated strong starters of order 21 there were about $1.8\%$ ``inconclusive'', 
of order 33 -- only 11 starters, and of order 39 --- none. The starter of order 39 in Example 2 of Sec.~\ref{sec:theorems}
was found by generating ``random'' strong starters until the first passing the test found;
it happened after more than 9 million iterations.

\smallskip
The reported experimental studies allow extensions and ramifications in various directions.
One interesting possible direction is to measure a correlation between triplication runtime, on the one hand, and the partition structure (the distribution of cardinalities) of weak sets for different values of key for the same base starter. 

	\section*{Acknowledgements}

	The authors are grateful to Prof.~N.~Shalaby of MUN for attracting our attention to this problem and many discussions 
    that inspired
    this research. We thank  
    Dr.~G.V.~Kalachev of MSU for introducing us to SAT solvers.
	

\section*{Appendix: Listing of Demo Example in Python}

\definecolor{darkred}{rgb}{0.6,0.0,0.0}
\definecolor{darkgreen}{rgb}{0,0.35,0}
\definecolor{lightblue}{rgb}{0.0,0.42,0.91}
\definecolor{orange}{rgb}{0.99,0.48,0.13}
\definecolor{grass}{rgb}{0.18,0.80,0.18}
\definecolor{pink}{rgb}{0.97,0.15,0.45}

\lstset{frame=tb,
   language=Python,
   aboveskip=3mm,
   belowskip=3mm,
   showstringspaces=false,
   columns=flexible,
   basicstyle={\small\ttfamily},
   numberstyle=\color{dkgray},
   commentstyle=\color{darkred},
   identifierstyle=\color{blue},
   stringstyle=\color{purple},
   keywordstyle=\color{darkgreen},
   breaklines=false,
   breakatwhitespace=true,
   tabsize=3,
   morekeywords={,as,assert,with,yield,self,True,False,}
}

\begin{lstlisting}
#!/usr/bin/env python3 

# This listing is Appendix in the paper: 
# "Constructing strong starters of orders 3p: triplication with SAT solver"
# by O. Ogandzhanyants, S. Sadov, M. Kondratieva
# May, 2025

#------------------------------------------------------------------------
# Solving "Sudoku mod3" problem for triplication of starters with z3 solver (Demo)
#-------------------------------------------------------------------------
# Loading functions from z3 solver library 
from z3 import Int, Solver, And, Distinct, sat
#-------------------------------------------------------------------------

# Defining the parameters

m=7    # Order of base strong starter
q=3    # q=(m-1)/2 :  number of pairs in base starter
n=21   # n=3*m : Order of triplicated starter
k=10   # k=(n-1)/2 = 3*q+1 : number of pairs in triplicated starter

#.........................................................................
# Defining the base starter and key
T=[(2,3),(4,6),(1,5)]  # base starter
key=1
print("Base starter=",T,"key=",key)
Sigma_p=[(1,1),\
        (2,3),(3,4),(5,6),\
        (4,6),(5,0),(2,4),\
        (1,5),(2,6),(3,0)]
# Verify:
Tplus=[((key+ab[0])% 7,(key+ab[1])% 7) for ab in T]
Tminus=[((key-ab[1])% 7,(key-ab[0])% 7) for ab in T]
assert (Sigma_p[0]==(key,key))
assert ([Sigma_p[i] for i in [1,4,7]]==T)
assert ([Sigma_p[i] for i in [2,5,8]]==Tplus)
assert ([Sigma_p[i] for i in [3,6,9]]==Tminus)

# Defining variables for modular Sudoku problem
# Main unknowns: elements of pairs (U,V) 
U=[Int(f'U_{i}') for i in range(k)]
V=[Int(f'V_{i}') for i in range(k)]

# Aux variables for differences D[i]=U[i]-V[i] mod 3 
D=[Int(f'D_{i}') for i in range(k)]

# Aux variables for sums S[i]=U[i]+V[i] mod 3 
#   (Only those S[i] with weak indexes will be essential.)
S=[Int(f'S_{i}') for i in range(k)]

# List of weak sets
weak_sets=[(2),(4,9),(1,5),(6,7)]

# List of indexes of weak pairs (=union of weak sets)
weak_indexes=[1,2,4,5,6,7,9]

# Dummy variable whose value will be set to 0
Z=Int(f'Z')

#.........................................................................
# Creating an instance of SAT solver from z3 library
Sudoku_solver=Solver()

#.........................................................................
# Setting up the constraints
    
# 0) Binding constraints:
    # Constraining dummy variable: Z=0
Sudoku_solver.add(Z==0)

    # Binding differences: D[i]==U[i]-V[i] mod 3        
for i in range(k):
    Sudoku_solver.add((U[i]-V[i]-D[i]) % 3 ==0)

    # Binding sums in weak pairs: S[i]==U[i]+V[i] mod 3        
for i in weak_indexes:
    Sudoku_solver.add((U[i]+V[i]-S[i]) % 3 ==0)

# - - - - - - - - - - - - - - - - - - - - - - - - - - - - - - 
# 1) Range Constraints: all values must be in {0,1,2}
for i in range(k):
    Sudoku_solver.add(And(0<= U[i], U[i]<=2))
    Sudoku_solver.add(And(0<= V[i], V[i]<=2))
    Sudoku_solver.add(And(0<= D[i], D[i]<=2))

for i in weak_indexes:
    Sudoku_solver.add(And(0<= S[i], S[i]<=2))

# - - - - - - - - - - - - - - - - - - - - - - - - - - - - - - 
# 2) Row Constraints: differences mod 3 in every row must be distinct
for i in range(q):  # in our example, i=0,1,2
    Sudoku_solver.add(Distinct(D[3*i+1: 3*i+4]))

# - - - - - - - - - - - - - - - - - - - - - - - - - - - - - - 
# 3) Weak Set Constraints: pair sums in weak sets must be distinct
    
    # Weak set {2} of type 1 is appended by dummy;  
Sudoku_solver.add(Distinct([S[2],Z])) # equivalent to S[2] != 0

    # Every weak set of size >=2 goes as is
Sudoku_solver.add(Distinct([S[4],S[9]]))  # For weak pair (4,9)
Sudoku_solver.add(Distinct([S[1],S[5]]))  # For weak pair (1,5)
Sudoku_solver.add(Distinct([S[6],S[7]]))  # For weak pair (6,7)

# - - - - - - - - - - - - - - - - - - - - - - - - - - - - - - 
# 4) Color Constraints: values in every monochrome triple must be distinct
Sudoku_solver.add(Distinct([V[5],V[9], Z]))  # Degenerate triple is appended by dummy
Sudoku_solver.add(Distinct([U[0], V[0], U[7]]))	
Sudoku_solver.add(Distinct([U[1], U[6], U[8]]))	
Sudoku_solver.add(Distinct([V[1], U[2], U[9]]))	
Sudoku_solver.add(Distinct([V[2], U[4], V[6]]))	
Sudoku_solver.add(Distinct([U[3], U[5], V[7]]))	
Sudoku_solver.add(Distinct([V[3], V[4], V[8]]))

#.........................................................................
#  Solving modular Sudoku problem
if Sudoku_solver.check() == sat:
    solution=Sudoku_solver.model()
    Sigma_3=[(solution.evaluate(U[i]).as_long(), solution.evaluate(V[i]).as_long())\
    			 for i in range(k)] 
    print("Sigma7=",Sigma_p)  
    print("Sigma3=",Sigma_3)           
        
#.........................................................................
# Finally, we merge the tables Sigma_p and Sigma_3 by CRT to get the new starter  

    # Conversion of pair of residues mod 7 and mod 3 to residue mod 21
    def CRT_merge(pairing_mod7, pairing_mod3, CRT_dictionary):
        return [tuple(CRT_dictionary[x,X] for x,X in zip(uv,UV))\
                    for uv,UV in zip(pairing_mod7,pairing_mod3)]

    # Building two starters from the same solution of Sudoku (Theorem 3)
      # dictionary to recover x from (x mod 7, x mod 3)
    CRT_dictionary1={(x%7,x%3): x for x in range(21)}  
    CRT_dictionary2={(x%7,(-x)%3): x for x in range(21)}  # x <- (x mod 7, -x mod 3)
    new_starter1=CRT_merge(Sigma_p, Sigma_3, CRT_dictionary1)
    new_starter2=CRT_merge(Sigma_p, Sigma_3, CRT_dictionary2)

    print ("Triplicated starters:")
    print (new_starter1)
    print (new_starter2)

#.........................................................................
else:
    print ("No solution")
    
# Output: (another run of the code may produce a different solution!)
 Base starter= [(2, 3), (4, 6), (1, 5)] key= 1
 Sigma7= [(1, 1), (2, 3), (3, 4), (5, 6), (4, 6), (5, 0), (2, 4), (1, 5), (2, 6), (3, 0)]
 Sigma3= [(1, 2), (1, 0), (2, 0), (1, 1), (2, 2), (0, 2), (0, 1), (0, 2), (2, 0), (1, 1)]
 Triplicated starters: 
 [(1,8), (16,3), (17,18), (19,13), (11,20), (12,14), (9,4), (15,5), (2,6), (10,7)]
 [(8,1), (2,3), (10,18), (5,20), (4,13), (12,7), (9,11), (15,19), (16,6), (17,14)]

\end{lstlisting}
    
\end{document}